\documentclass[11pt]{amsart}
\usepackage{amssymb}
\usepackage{bm}
\usepackage{graphicx}
\usepackage[centertags]{amsmath}
\usepackage{amsfonts}
\usepackage{amsthm}
\newtheorem{thm}{Theorem}
\newtheorem{cor}[thm]{Corollary}
\newtheorem{lem}[thm]{Lemma}

\newtheorem{prop}[thm]{Proposition}
\newtheorem{claim}[thm]{Claim}
\newtheorem{fact}[thm]{Fact}
\newtheorem{defn}[thm]{Definition}
\theoremstyle{definition}

\newtheorem{question}{Question}

\newcommand{\nn}{\mathbb{N}}
\newcommand{\ee}{\varepsilon}

\newcommand{\xx}{\mathbf{x}}
\newcommand{\overp}{\overline{p}}

\newcommand{\coll}{\mathrm{Col}}
\newcommand{\iii}{\mathcal{I}}
\newcommand{\fff}{\mathcal{F}}
\newcommand{\ddd}{\mathcal{D}}
\newcommand{\uuu}{\mathcal{U}}
\newcommand{\vvv}{\mathcal{V}}
\newcommand{\col}{\coll(\lambda,<\kappa)}
\newcommand{\al}{\alpha}
\newcommand{\be}{\beta}
\newcommand{\de}{\delta}

\newcommand{\la}{\lambda}
\newcommand{\ka}{\kappa}
\newcommand{\ga}{\gamma}

\newcommand{\om}{\omega}
\newcommand{\pl}{\mathrm{Pl}}
\newcommand{\expp}{\mathrm{exp}}
\newcommand{\con}{\smallfrown}
\newcommand{\vep}{\varepsilon}

\newcommand{\rest}{\upharpoonright}

\newcommand{\mc}[1]{\mathcal{#1}}
\newcommand{\mm}[1]{\mathrm{#1}}

\newcommand{\set}[2]{ \{ {#1}\,:\,{#2} \} }
\newcommand{\sset}{\subseteq}

\newcommand{\nrm}[1]{\|#1\|}

\newcommand{\N}{{\mathbb N}}

\begin{document}

\title[Unconditional basic sequences]{Unconditional basic sequences
in spaces of large density}
\author{Pandelis Dodos, Jordi Lopez-Abad and Stevo Todorcevic}

\address{Universit\'{e} Pierre et Marie Curie - Paris 6, Equipe d'Analyse
Fonctionnelle, Bo\^{i}te 186, 4 place Jussieu, 75252 Paris Cedex 05, France.}
\email{pdodos@math.ntua.gr}

\address{Universit\'{e} Denis Diderot - Paris 7, Equipe de Logique
Math\'{e}matiques, 2 place Jussieu, 72521 Paris Cedex 05, France.}
\email{abad@logique.jussieu.fr}

\address{Universit\'{e} Denis Diderot - Paris 7, C.N.R.S., UMR 7056, 2 place Jussieu
- Case 7012, 72521 Paris Cedex 05, France}
\address{and}
\address{Department of Mathematics, University of Toronto, Toronto, Canada, M5S 2E4.}
\email{stevo@math.toronto.edu}

\footnotetext[1]{2000 \textit{Mathematics Subject Classification}. Primary 46B03, 03E35;
Secondary 03E02, 03E55, 46B26, 46A35.}
\footnotetext[2]{\textit{Key words}: unconditional basic sequence,
non-separable Banach spaces, separable quotient problem, forcing,
polarized Ramsey, strongly compact cardinal. }

\maketitle


\begin{abstract}
We study the problem of the existence of unconditional basic sequences in Banach spaces of high
density. We show, in particular, the relative consistency with GCH of the statement that every
Banach space of density $\aleph_\om$ contains an unconditional basic sequence.
\end{abstract}


\section{Introduction}

In this paper we study particular instances of the general
unconditional basic sequence problem asking under which conditions a
given Banach space must contain an infinite unconditional basic
sequence (see \cite[page 27]{LT}). We chose to study instances of
the problem for Banach spaces of large densities exposing thus its
connections with large-cardinal axioms of set theory. The first
paper on this line of research is a well-known paper of J. Ketonen
\cite{Ke} which shows that if a density of a given Banach space $E$
is greater or equal to the $\omega$-Erd\H{o}s cardinal (usually
denoted as $\kappa(\omega)$, see Section 2.2), then $E$ contains an
infinite unconditional basic sequence. More precisely, let
$\mathfrak{nc}$ be the minimal cardinal $\lambda$ such that every
Banach space of density at least $\lambda$ contains an infinite
unconditional basic sequence. Then Ketonen's result can be restated
as follows.
\begin{thm}[\cite{Ke}]
$\kappa(\omega)\geq\mathfrak{nc}.$
\end{thm}
Since $\kappa(\omega)$ is a considerably large cardinal (strongly
inaccessible and more) one would like to determine is
$\mathfrak{nc}$  really a large cardinal or not, and, of course at
some point one would also like to determine the exact value of
this cardinal. Unfortunately, there are not too many results in
the literature that would point out towards lower bounds for this
cardinal. In fact, the largest known lower bound for
$\mathfrak{nc}$ is given by S. A. Argyros and A. Tolias \cite{AT}
who showed that $\mathfrak{nc}>2^{\aleph_0}.$ So in particular the
following problem is widely open.
\begin{question}
Is ${\rm exp}_\omega(\aleph_0),$ any of the finite-tower exponents
${\rm exp}_n(\aleph_0)$, or any of their $\omega$-successors ${\rm
exp}_n(\aleph_0)^{+\omega}$ an upper bound of $\mathfrak{nc}$? In
particular, does $(2^{\aleph_0})^{+\omega}\geq\mathfrak{nc}$ hold?
\end{question}
Our first result shows that ${\rm exp}_\omega(\aleph_0)$ is not
such a bad candidate for an upper bound of $\mathfrak{nc}.$
\begin{thm}
\label{ithm1} The inequality ${\rm
exp}_\omega(\aleph_0)\geq\mathfrak{nc}$ is a statement that is
consistent relative to the consistency of infinitely many strongly
compact cardinals.
\end{thm}
The consistency proof relies heavily on a Ramsey-theoretic
property of $\mathrm{exp}_\om(\aleph_0)$ established in a previous
work of S. Shelah \cite{Sh2} (see also \cite{Mi}). One can also
arrange the joint consistency of GCH and the inequality ${\rm
exp}_\omega(\aleph_0)=\aleph_\omega\geq\mathfrak{nc}$. Combining
this with a well known result of J. N. Hagler and W. B. Johnson
\cite{HJ}, we get the following information about the famous
separable quotient problem.
\begin{cor}
It is relatively consistent that every Banach space of density at
least $\aleph_\omega$ has a separable quotient with an
unconditional basis.
\end{cor}
The analysis given in this paper together with some known results
from Banach space theory suggest, in particular, that by
restricting the class of Banach spaces to, say, reflexive, or more
generally weakly compactly generated Banach spaces, one might get
different answers about the size of the corresponding cardinal
numbers $\mathfrak{nc}_{\rm rfl}$ and $\mathfrak{nc}_{\rm wcg},$
respectively. To describe this difference it will be convenient to
introduce yet another natural cardinal characteristic
$\mathfrak{nc}_{\rm seq}$, the minimal cardinal $\theta$ such that
every normalized weakly null sequence $(x_\al:\al<\theta)$ in some
Banach space $E$ has a subsequence which is unconditional. Clearly
$\mathfrak{nc}_{\rm rfl}\leq \mathfrak{nc}_{\rm wcg}$ while by the
Amir-Lindenstrauss theorem \cite{AL} we see that
$\mathfrak{nc}_{\rm wcg}\leq \mathfrak{nc}_{\rm seq}$. The first
known lower bound on these cardinal is due to B. Maurey and H. P.
Rosenthal \cite{MR} who showed that $\mathfrak{nc}_{\rm
seq}>\aleph_0$, though considerably deeper is the lower bound of
W. T. Gowers and B. Maurey \cite{GM} who showed that in fact
$\mathfrak{nc}_{\rm rfl}>\aleph_0$. The largest known lower bound
on these cardinals is given in \cite{ALT} who showed that
$\mathfrak{nc}_{\rm rfl}>\aleph_1.$ This suggests the following
question.
\begin{question}
Is $\aleph_\omega$ or any of the finite successors $\aleph_n$
$(n\geq 2)$ an upper bound on any of the three cardinals
$\mathfrak{nc}_{\rm seq},$ $\mathfrak{nc}_{\rm rfl},$ or
$\mathfrak{nc}_{\rm wcg}$?
\end{question}
That $\aleph_\omega$ is not such a bad choice for an upper bound
of $\mathfrak{nc}_{\rm seq}$ may be seen from our second result.
\begin{thm}
The inequality  $\aleph_\om\geq\mathfrak{nc}_{\rm seq}$ is a
statement that is consistent relative to the consistency of a
single measurable cardinal.
\end{thm}
Thus, the consistency proof uses a considerably weaker assumption
from that used in Theorem \ref{ithm1}. It relies on two
Ramsey-theoretic principles, one established by P. Koepke
\cite{Ko} and the other by C. A. Di Prisco and S. Todorcevic
\cite{DT}. It also gives the joint consistency of the GCH and the
cardinal inequality  $\aleph_\om\geq\mathfrak{nc}_{\rm seq}$.

\section{Preliminaries}

Our Banach space theoretic and set theoretic terminology and
notation are standard and follow \cite{LT} and \cite{Ku},
respectively. We will consider only real Banach spaces though,
using essentially the same arguments, one notices that all our
results are valid for complex Banach spaces as well.

Since in this note we are concerned with the problem of the
existence of unconditional basic sequences in Banach spaces
of high density, let us introduce the following cardinal invariants
related to the version of the unconditional basic sequence problem
that we study here.
\begin{defn}
\label{CardinalInvariants} Let $\mathfrak{nc}$, $\mathfrak{nc}_{\rm wcg}$, $\mathfrak{nc}_{\rm rfl}$
and $\mathfrak{nc}_{\rm seq}$ be defined as follows.
\begin{enumerate}
\item[(1)] $\mathfrak{nc}$ is the minimal cardinal $\lambda$ such that every Banach space of density
$\lambda$ contains an unconditional basic sequence.
\item[(2)] $\mathfrak{nc}_{\rm wcg}$ (respectively, $\mathfrak{nc}_{\rm rfl}$) is the minimal cardinal
$\lambda$ such that every weakly compactly generated (respectively, reflexive) Banach space of density
$\lambda$ contains an unconditional basic sequence.
\item[(3)] $\mathfrak{nc}_{\rm seq}$ is the minimal cardinal $\lambda$ such that every normalized weakly
null sequence $(x_\al:\al<\lambda)$ in a Banach space $E$ has a subsequence which is unconditional.
\end{enumerate}
\end{defn}
Let us now recall some standard set theoretic notions that will be used throughout the paper.

\subsection{Ideals on Fields of Sets}

Let $X$ be a non-empty set. An \textit{ideal} $\iii$ on $X$
is a collection of subsets of $X$ satisfying the following conditions.
\begin{enumerate}
\item[(i)] If $A\in\iii$ and $B\subseteq A$, then $B\in\iii$.
\item[(ii)] If $A,B\in\iii$, then $A\cup B\in\iii$.
\end{enumerate}
If $\iii$ is an ideal on $X$ and $\ka$ is a cardinal, then
we say that $\iii$ is \textit{$\ka$-complete} if
for every $\la<\ka$ and every sequence $(A_\xi:\xi<\la)$ in
$\iii$ we have $\bigcup_{\xi<\la} A_\xi\in\iii$.

A subset $A$ of $X$ is said to be \textit{positive} with respect
to an ideal $\iii$ if $A\notin \iii$. The set of all positive
sets with respect to $\iii$ is denoted by $\iii^+$. If $\ddd$ is
a subset of $\iii^+$ and $\ka$ is a cardinal, then we say that
$\ddd$ is \textit{$\ka$-closed in $\iii^+$} if for every $\lambda<\ka$
and every decreasing sequence $(D_\xi:\xi<\lambda)$ in $\ddd$ we have
$\bigcap_{\xi<\ka} D_\xi\in \iii^+$. We also say that such a set
$\ddd$ is \textit{dense} in $\iii^+$ if for every $A\in\iii^+$
there exists $D\in\ddd$ with $D\subseteq A$.

If $\fff$ is a filter on $X$, then the family $\set{X\setminus A}{A\in\fff}$
is an ideal. Having in mind this correspondence, we will continue to
use the above terminology for the filter $\fff$. Notice that
if the given filter is actually an ultrafilter $\uuu$,
then, setting $\iii=\mathcal{P}(X)\setminus \uuu$, we have
that $\iii^+=\uuu$.

\subsection{Large Cardinals}

Let $\theta$ be a cardinal.
\begin{enumerate}
\item[(a)] $\theta$ is said to be \emph{inaccessible} if it is regular and strong limit; that is,
$2^\la<\theta$ for every $\la<\theta$.
\item[(b)] $\theta$ is said to be \emph{$0$-Mahlo} if it is inaccessible. In general, for an ordinal
$\al$, $\theta$ is said to be \emph{$\al$-Mahlo} if for every $\be<\al$ and every closed and unbounded
subset $C$ of $\theta$ there is a $\be$-Mahlo cardinal $\lambda$ in $C$.
\item[(c)] An \emph{$\al$-Erd\H{o}s} cardinal, usually denoted by $\kappa(\al)$ if exists, is the
minimal cardinal $\lambda$ such that $\lambda\rightarrow (\al)^{<\omega}_2$; that is, $\lambda$ is the
least cardinal with the property that for every coloring $c:[\lambda]^{<\omega}\rightarrow 2$ there
is $H\subseteq \lambda$ of order-type $\al$ such that $c$ is constant on $[H]^n$ for every $n<\omega$.
A cardinal $\lambda$ that is $\lambda$-Erd\H{o}s (in other words, a cardinal $\lambda$ which has the
partition property $\lambda\rightarrow (\lambda)^{<\omega}_2$) is called a \emph{Ramsey} cardinal.
\item[(d)] $\theta$ is said to be \emph{measurable} if there exists a $\ka$-complete normal ultrafilter
$\mc U$ on $\ka$. Looking at the ultrapower of the universe using $\mc U$ one can observe that the set
$\set{\la<\theta}{\la \text{ is inaccessible}}$ belongs to $\mathcal U$. Similarly, one shows that sets
$\set{\la<\theta}{\la \text{ is } \lambda\text{-Mahlo}}$ and $\set{\la<\theta}{\la \text{ is Ramsey}}$
belong to $\mc U$.
\item[(e)] $\theta$ is said to be \emph{strongly compact} if every $\ka$-complete filter can be extended
to a $\ka$-complete ultrafilter.
\end{enumerate}
Finally, for every cardinal $\ka$ and every $n\in\om$ we define recursively the cardinal
$\expp_n(\ka)$ by the rule $\expp_0(\ka)=\ka$ and $\expp_{n+1}(\ka)=2^{\expp_n(\ka)}$.

\subsection{The L\'{e}vy Collapse}

Let $\la$ be a regular infinite cardinal and let $\ka>\la$ be an
inaccessible cardinal. By $\coll(\la,<\ka)$ we shall denote the set
of all partial mappings $p$ satisfying the following.
\begin{enumerate}
\item[(i)] $\mathrm{dom}(p)\subseteq \la\times\ka$ and
$\mathrm{range}(p)\subseteq\ka$.
\item[(ii)] $|p|<\la$.
\item[(iii)] For every $(\al,\be)\in\mathrm{dom}(p)$ with $\be>0$
we have $p(\al,\be)<\be$.
\end{enumerate}
We equip the set $\coll(\la,<\ka)$ with the partial order $\leq$ defined by
\[ p\leq q \Leftrightarrow \mathrm{dom}(q)\subseteq \mathrm{dom}(p) \text{ and }
p\upharpoonright\mathrm{dom}(q)=q.\] If $p$ and $q$ is a pair in $\coll(\la,<\ka)$, then by $p\ \|\ q$ we
denote the fact that $p$ and $q$ are \textit{compatible} (i.e. there exists $r$ in $\coll(\la,<\ka)$ with
$r\leq p$ and $r\leq q$), while by $p\perp q$ we denote the fact that $p$ and $q$ are \textit{incompatible}.

We will need the following well-known properties of the L\'{e}vy
collapse (see, for instance, \cite{Ka}). In what follows, $G$ will
be a $\coll(\la,<\ka)$-generic filter.
\begin{enumerate}
\item[(a)] The generic filter $G$ is $\la$-complete (this is a
consequence of the fact that the forcing $\coll(\la,<\ka)$ is $\la$-closed).
\item[(b)] $\coll(\la,<\ka)$ has the $\ka$-cc (this follows
from the fact that the cardinal $\ka$ is inaccessible).
\item[(c)] In $V[G]$, we have $\ka=\la^+$.
\item[(d)] In $V[G]$, the sets $^{\ka}{2}$ and $^{\ka}{2}\cap V$
are equipotent.
\end{enumerate}
Finally, let us introduce some pieces of notation (actually,
these pieces of notation will be used only in \S 5). For every
$p\in\coll(\la,<\ka)$ and every $\al<\ka$ by $p\rest \al$ we
shall denote the restriction of the partial map $p$ to
$\mathrm{dom}(p)\cap(\la\times \al)$. Moreover, for
every $p\in\coll(\la,<\ka)$ we let
$(\mathrm{dom}(p))_1=\set{\al<\ka}{\exists \xi<\la \text{ with } (\xi,\al)
\in\mathrm{dom}(p)}$.


\section{A polarized partition relation}

The purpose of this section is to analyze the following partition
property, a variation of a partition property originally appearing
in the problem lists of P. Erd\H{o}s and A. Hajnal \cite{EH1},
\cite[Problem 29]{EH2} (see also \cite{Sh2}).
\begin{defn}
\label{oriutirt} Let $\ka$ be a cardinal and $d\in\om$ with $d\geq 1$. By $\mathrm{Pl}_d(\ka)$ we
shall denote the combinatorial principle asserting that for every coloring
$c:\big[\,[\ka]^d\,\big]^{<\om}\to\om$ there exists a sequence $(\xx_n)$ of infinite disjoint
subsets of $\ka$ such that for every $m\in\om$ the restriction $c\rest \prod_{n=0}^m [\xx_n]^d$ is
constant.
\end{defn}
Clearly property $\mathrm{Pl}_d(\ka)$ implies property $\mathrm{Pl}_{d'}(\ka)$ for any cardinal $\ka$
and any pair $d, d'\in\om$ with $d\geq d'\geq 1$. From known results one can easily deduce that the
principle $\pl_{d}({\rm exp}_{d-1}(\aleph_0)^{+n})$ is false for every $n\in\om$ and every integer
$d\geq 1$ (see, for instance, \cite{EHMR}, \cite{CDPM} and \cite{DT}). Thus, the minimal cardinal
$\kappa$ for which $\mathrm{Pl}_d(\ka)$ could possibly be true is ${\rm exp}_{d-1}(\aleph_0)^{+\omega}$.
Indeed, C. A. Di Prisco and S. Todorcevic \cite{DT} have established the consistency of
$\pl_1(\aleph_\omega)$ relative the consistency of a single measurable cardinal, an assumption that
also happens to be optimal. On the other hand, S. Shelah \cite{Sh2} was able to establish that GCH
and principles $\mathrm{Pl}_d(\aleph_\om)$ $(d\geq 1)$ are jointly consistent, relative to the consistency
of GCH and the existence of an infinite sequence of strongly compact cardinals.

Our aim in this section is to present a consistency proof of $\pl_2\big(\expp_\om(\aleph_0)\big)$. We
shall treat the colorings in Definition \ref{oriutirt} using an iteration of the following lemma whose
proof (given in \S 5), while it relies heavily on an idea of S. Shelah \cite{Sh2}, it exposes certain
features (the ideal $\iii$ and the sufficiently complete dense subset $\ddd$ of its quotient), not
explicitly found in \cite{Sh2}, that are likely to find application beyond the scope of our present paper.
\begin{lem}
\label{MainLemma} Suppose that $\ka$ is a strongly compact cardinal and that $\lambda<\kappa$ is an
infinite regular cardinal. Let $G$ be a $\coll(\lambda,<\kappa)$-generic filter over $V$. Then, in
$V[G]$, for every integer $d\geq 1$ there exists an ideal $\iii_d$ on $[(\expp_d(\ka))^+]^\omega$
and a subset $\ddd_d$ of $\iii_d^+$ such that the following are satisfied.
\begin{enumerate}
\item[(1)] $\iii_d$ is $\kappa$-complete.
\item[(2)] $\ddd_d$ is dense in $\iii_d^+$ and is $\la$-closed in $\iii_d^+$.
\item[(3)] For every $\mu<\ka$, every coloring $c:[(\expp_d(\ka))^+]^{d+1}\to\mu$ and every set
$A\in\iii^+_d$ there exist a color $\xi<\mu$ and an element $D\in\ddd_d$ with $D\sset A$ and such
that for every $\xx\in D$ the restriction $c\upharpoonright [\xx]^{d+1}$ is constantly equal to $\xi$.
\end{enumerate}
\end{lem}
To state our next result it is convenient to introduce a sequence $(\Theta_n)$ of cardinals, defined
recursively by the rule
\begin{equation}
\label{etheta} \Theta_0=\aleph_0 \ \text{ and } \ \Theta_{n+1}=\big( 2^{(2^{\Theta_n})^+}\big)^{++}.
\end{equation}
Notice that the sequence $(\Theta_n)$ is strictly increasing and that
\[ \expp_n(\aleph_0)<\Theta_n\ \leq \expp_{5n}(\aleph_0) \]
for every $n\in\om$ with $n\geq 1$. Hence, $\sup\{\Theta_n\,:\,n\in\om\}=\expp_\om(\aleph_0)$.
In particular, if GCH holds, then $\Theta_n=\aleph_{5n}$ for every $n\in\om$.
\begin{thm}
\label{maincombcor} Suppose that $(\ka_n)$ is a strictly increasing sequence of strongly compact
cardinals with $\ka_0=\aleph_0$. For every $n\in\om$ set $\lambda_n=\big(2^{(2^{\ka_n})^+}\big)^+$. Let
\[ \mathbb P=\bigotimes_{n\in\om}\mathrm{Col}(\la_n,<\ka_{n+1})\]
be the iteration of the sequence of L\'{e}vy collapses. Let $G$ be a $\mathbb{P}$-generic filter over $V$.
Then, in $V[G]$, for every $n\in\om$ we have $\ka_n=\Theta_n$ and there exist an ideal $\iii_n$ on
$[(2^{\Theta_{n+1}})^+]^\om$ and a subset $\ddd_n$ of $\iii_n^+$ such that the following are satisfied.
\begin{enumerate}
\item[(P1)] $\iii_n$ is $\Theta_{n+1}$-complete.
\item[(P2)] $\ddd_n$ is $(<\Theta_{n+1})$-closed in $\iii_n^+$; that is, $\ddd_n$ is $\mu$-closed in
$\iii_n^+$ for every $\mu<\Theta_{n+1}$.
\item[(P3)] For every $\mu<\Theta_{n+1}$, every coloring $c:[(2^{\Theta_{n+1}})^+]^2\to\mu$ and every
$A\in\iii_n^+$ there exist a color $\xi<\mu$ and an element $D\in\ddd_n$ with $D\subseteq A$ and such
that for every $\xx\in D$ the restriction $c\rest [\xx]^2$ is constantly equal to $\xi$.
\end{enumerate}
Moreover, if GCH holds in $V$, then GCH also holds in $V[G]$.
\end{thm}
\begin{proof}
Let $n\in\om$ arbitrary. Let $G_n$ be the restriction of $G$ to the finite iteration
\[ \mathbb P_n=\bigotimes_{m<n}\mathrm{Col}(\la_m,<\ka_{m+1}). \]
Notice, first, that the small forcing extension $V[G_n]$ preserves the strong compactness of $\ka_{n+1}$.
This fact follows immediately from the elementary-embedding characterization of strong compactness
(see \cite[Theorem 22.17]{Ka}). Working in $V[G_n]$ and applying Lemma \ref{MainLemma} for $d=1$,
we see that the intermediate forcing extension $V[G_{n+1}]$ has the ideal $\iii_n$ whose quotient has
properties (P1), (P2) and (P3) described in Lemma \ref{MainLemma}. Working still in the intermediate
forcing extension $V[G_{n+1}]$, we see that the rest of the forcing
\[ \mathbb P^{n+1}=\bigotimes_{n<m<\omega}\mathrm{Col}(\la_m,<\ka_{m+1})\]
is $\lambda_{n+1}$-closed, and so, in particular, it adds no new subsets to the index set on which the
ideal $\iii_n$ lives. Therefore, properties (P1), (P2) and (P3) of the quotient of $\iii_n$ are preserved
in $V[G]$. Since $n$ was arbitrary, the proof is completed.
\end{proof}
The final step towards our proof of the consistency of $\pl_2\big(\expp_\om(\aleph_0)\big)$
is included in the following, purely combinatorial, result.
\begin{prop}
\label{first} Let $(\Theta_n)$ be the sequence of cardinals defined in (\ref{etheta}) above.
Suppose that for every $n\in\om$ there exist an ideal $\iii_n$ on $[(2^{\Theta_{n+1}})^+]^\om$
and a subset $\ddd_n$ of $\iii_n^+$ which satisfy properties (P1), (P2) and (P3) described in
Theorem \ref{maincombcor}. Then, the principle $\mathrm{Pl}_2\big(\expp_\om(\aleph_0)\big)$ holds.
\end{prop}
\begin{proof}
The proof is based on the following claim.
\begin{claim}
\label{claimnew} Let $n\in\om$. Let also $c: \prod_{i=0}^{n} [(2^{\Theta_{i+1}})^+]^2 \to \om$ be a
coloring and $(D_i)_{i=0}^n\in \prod_{i=0}^n\mathcal D_i$. Then, there exist $(E_i)_{i=0}^n\in
\prod_{i=0}^n\mathcal D_i\rest D_i$ and a color $m_0\in\om$ such that for every $(\xx_i)_{i=0}^n\in
\prod_{i=0}^nE_i$ the restriction $c\rest \prod_{i=0}^n [\xx_i]^2$ is constantly equal to $m_0$.
\end{claim}
\begin{proof}[Proof of Claim \ref{claimnew}]
By induction on $n$. The case $n=0$ is an immediate consequence of property (P3) in Theorem
\ref{maincombcor}. So, let $n\in\om$ with $n\geq 1$ and assume that the result has been proved for all
$k\in\om$ with $k<n$. Fix a coloring $c:\prod_{i=0}^n [(2^{\Theta_{i+1}})^+]^2 \to \om$. Fix also
$(D_i)_{i=0}^n\in \prod_{i=0}^n\mathcal D_i$ and let
\[ \mathcal{F}=\{ f:\prod_{i=0}^{n-1} [(2^{\Theta_{i+1}})^+]^2 \to \om: f \text{ is a coloring}\}.\]
Notice that $|\mathcal{F}|=2^{(2^{\Theta_n})^+}$, and so, $|\mathcal{F}|<\Theta_{n+1}$. We define a
coloring $d:[(2^{\Theta_{n+1}})^+]^2 \to \mathcal{F}$ by the rule
$d\big(\{\al,\be\}\big)(\bar{s})=c\big(\bar{s}^{\con}\{\al,\be\}\big)$
for every $\bar{s}\in \prod_{i=0}^{n-1}[(2^{\Theta_{i+1}})^+]^2$. By (P3) in Theorem \ref{maincombcor},
there exist $E_n\in \mathcal D_n \rest D_n$ and $f_0\in\mathcal{F}$ such that for every $\xx\in E_n$
the restriction $d\rest [\xx]^2$ is constantly equal to $f_0$. The result now follows by applying our
inductive hypothesis to the coloring $f_0$.
\end{proof}
By Claim \ref{claimnew} and the fact that every $\mathcal D_n$ is $\sigma$-closed (property (P2) in
Theorem \ref{maincombcor}), the proof of Proposition \ref{first} is completed.
\end{proof}
As a consequence of the previous analysis we get the following.
\begin{cor}[\cite{Sh2}]
\label{shelah} Suppose that in our universe $V$ there exists a strictly increasing sequence $(\ka_n)$
of strongly compact cardinals with $\ka_0=\aleph_0$. Then, there is a forcing extension of $V$ in
which the principle $\mathrm{Pl}_2\big(\expp_\om(\aleph_0)\big)$ holds. Moreover, if GCH holds in $V$,
then GCH also holds in the extension.
\end{cor}
\begin{proof}
Follows by Theorem \ref{maincombcor} and Proposition \ref{first}.
\end{proof}
Clearly, in the forcing extension obtained above the combinatorial principle
$\pl_1\big(\expp_\om(\aleph_0)\big)$ holds as well. However, as we have already indicated, one
can obtain the consistency of $\pl_1\big(\expp_\om(\aleph_0)\big)$ using a considerably weaker
(and, in fact, optimal) large-cardinal assumption from the one used for
$\pl_2\big(\expp_\om(\aleph_0)\big)$. More precisely, we have the following.
\begin{thm}[\cite{DT}]
\label{thmnewnew} Assume the existence of a measurable cardinal. Then, there is a forcing extension in
which GCH and $\mm{Pl}_1(\aleph_\om)$ hold.
\end{thm}
In our applications in Banach space theory we will use only the combinatorial principles
$\mm{Pl}_2\big(\expp_\om(\aleph_0)\big)$ and $\mm{Pl}_1(\aleph_\om)$. We would like, however, to record
the following higher-dimensional analogues of Theorem \ref{maincombcor}, Proposition \ref{first}
and Corollary \ref{shelah} respectively. It appears that these analogues can be used in a variety
of problems of combinatorial flavor. Their proofs are straightforward adaptations of our
previous arguments (we leave the details to the interested reader).
\begin{thm}
\label{maincombcorvar} Suppose that $(\ka_n)$ is a strictly increasing sequence of strongly compact
cardinals with $\ka_0=\aleph_0$. For every $n\in\om$ we can choose cardinals $\lambda_n, \theta_n\in
\big[ \kappa_n, {\rm exp}_\omega(\kappa_n)\big)$ in such a way that, if we let
\[ \mathbb P=\bigotimes_{n\in\om}\mathrm{Col}(\la_n,<\ka_{n+1})\]
be the iteration of the sequence of L\'{e}vy collapses and if we choose $G$ to be a $\mathbb{P}$-generic
filter over $V$, then, in $V[G]$, we have
\[ \sup_{n\in\om}\kappa_n=\sup_{n\in\om}\lambda_n=\sup_{n\in\om}\theta_n={\rm exp}_\omega(\aleph_0)\]
and for every $n\in\om$ there exist an ideal $\iii_n$ on $[\theta_{n+1}]^\om$ and a subset $\ddd_n$ of
$\iii_n^+$ such that the following are satisfied.
\begin{enumerate}
\item[(P1)] $\iii_n$ is $\kappa_{n+1}$-complete.
\item[(P2)] $\ddd_n$ is $(<\lambda_{n})$-closed in $\iii_n^+$; that is, $\ddd_n$ is $\mu$-closed in
$\iii_n^+$ for every $\mu<\lambda_{n}$.
\item[(P3)] For every $\mu<\kappa_{n+1}$, every coloring $c:[\theta_{n+1}]^{n+1}\to\mu$ and every
$A\in\iii_n^+$ there exist a color $\xi<\mu$ and an element $D\in\ddd_n$ with $D\subseteq A$ and such
that for every $\xx\in D$ the restriction $c\rest [\xx]^{n+1}$ is constantly equal to $\xi$.
\end{enumerate}
Moreover, if GCH holds in $V$, then GCH also holds in the $V[G]$.
\end{thm}
\begin{prop}
\label{firstvar} Suppose $(\ka_n)$, $(\lambda_n)$ and $(\theta_n)$ are strictly increasing sequences of
regular cardinal that all converge to ${\rm exp}_\omega(\aleph_0)$. Suppose further that for every
$n\in\om$ there exist an ideal $\iii_n$ on $[\theta_{n+1}]^\om$ and a subset $\ddd_n$ of $\iii_n^+$ which
satisfy properties (P1), (P2) and (P3) of Theorem \ref{maincombcorvar}. Then the principle
$\mathrm{Pl}_d\big(\expp_\om(\aleph_0)\big)$ holds for every integer $d\geq 1$.
\end{prop}
\begin{cor}[\cite{Sh2}]
\label{shelahvar} Suppose that in our universe $V$ there exists a strictly increasing sequence $(\ka_n)$
of strongly compact cardinals with $\ka_0=\aleph_0$. Then, there is a forcing extension of $V$ in
which the principle $\mathrm{Pl}_d\big(\expp_\om(\aleph_0)\big)$ holds for every integer $d\geq 1$.
Moreover, if GCH holds in $V$, then GCH also holds in the forcing extension.
\end{cor}


\section{Banach space implications}

Let us recall that a sequence $(x_n)$ in a Banach space $E$ is said
to be \textit{$C$-unconditional}, where $C\geq 1$, if for every pair
$F$ and $G$ of non-empty finite subsets of $\om$ with $F\subseteq G$
and every choice $(a_n)_{n\in G}$ of scalars we have
\[ \| \sum_{n\in F} a_n x_n \| \leq C \cdot \|\sum_{n\in G} a_n x_n \|.\]
This main result in this section is the following.
\begin{thm}
\label{mainappli}
Let $\ka$ be a cardinal and assume that property $\mathrm{Pl}_2(\ka)$
holds (see Definition \ref{oriutirt}). Then every Banach space $E$ not
containing $\ell_1$ and of density $\ka$ contains an $1$-unconditional
basic sequence.

In particular, if $E$ is any Banach space of density $\ka$, then for every
$\ee>0$ the space $E$ contains an $(1+\ee)$-unconditional basic sequence.
\end{thm}
Combining Corollary \ref{shelah} with Theorem \ref{mainappli}, we
get the following corollaries.
\begin{cor}
\label{previouscorollary} It is consistent relative the existence of an infinite sequence of strongly
compact cardinals that for every $\ee>0$ and every Banach space $E$ of density at least
$\expp_\om(\aleph_0)$, the space $E$ contains an $(1+\ee)$-unconditional basic sequence. Moreover,
this statement is consistent with GCH.
\end{cor}
\begin{proof}
Follows immediately by Corollary \ref{shelah} and Theorem \ref{mainappli}.
\end{proof}
\begin{cor}
It is consistent relative to the existence of an infinite sequence of strongly compact cardinals that
every Banach space of density at least $\expp_\om(\aleph_0)$ has a separable quotient with an
unconditional basis. Moreover, this statement is consistent with GCH.
\end{cor}
\begin{proof}
A well-known consequence of a result due to J. N. Hagler and W. B.
Johnson \cite{HJ} asserts that if $E$ is a Banach space such that
$E^*$ has an unconditional basic sequence, then $E$ has a separable
quotient with an unconditional basis (see also \cite[Proposition
16]{ADK}). Noticing that the density of the dual $E^*$ of a Banach
space $E$ is at least as big as the density of $E$, the result
follows by Corollary \ref{previouscorollary}.
\end{proof}
The section is organized as follows. In \S 4.1 we give the proof of
Theorem \ref{mainappli}, while in \S 4.2 we present its ``sequential" version.
Two proofs of this version are given, each of which is based on a different
combinatorial principle.


\subsection{Proof of Theorem \ref{mainappli}}

We start with the following lemma, which is essentially a multi-dimensional version of Odell's
Schreier unconditionality theorem \cite{Od}.
\begin{lem}\label{wnullunco}
Let $E$ be a Banach space, $m\in\om$ with $m\geq 1$ and $\ee>0$.
For every $i\in\{0,\dots,m\}$ let $(x_n^i)$ be a normalized weakly null
sequence in the space $E$. Then, there exists an infinite subset $L$ of $\om$
such that for every $\{n_0<\dots <n_m\}\sset L$ the sequence
$(x_{n_i}^i)_{i=0}^m$ is $(1+\ee)$-unconditional.
\end{lem}
\begin{proof}
The first step towards the proof of the lemma is included in the following claim. It shows that, by
passing to an infinite subset of $\om$, we may assume that for every $\{n_0<\dots<n_m\}\in [\nn]^{m+1}$
the finite sequence $(x^i_{n_i})_{i=0}^m$ is a particularly well behaved Schauder basic sequence.
\begin{claim}\label{o4ut44iui4}
For every $\vep>0$ there exists an infinite subset $M$ of $\om$ such that for every
$\{n_0<\dots <n_m\}\sset M$ the sequence $(x_{n_i}^i)_{i=0}^m$ is an $(1+\vep)$-Schauder
basic sequence.
\end{claim}
\begin{proof}[Proof of Claim \ref{o4ut44iui4}]
We define a coloring $\mathcal{B}:[\N]^{m+1}\to 2$ as follows.
Let $s=\{n_0<\dots<n_m\}\in [\nn]^{m+1}$ arbitrary. If $(x_{n_i}^i)_{i=0}^m$ is an $(1+\vep)$-Schauder
basic sequence, then we set $\mathcal{B}(s)=0$; otherwise we set $\mathcal{B}(s)=1$. By Ramsey's
theorem, there exist an infinite subset $M$ of $\om$ and $c\in\{0,1\}$ such that
$\mathcal{B}\upharpoonright [M]^{m+1}$ is constantly equal to $c$. Using Mazur's classical
procedure for selecting Schauder basic sequences (see, for instance, \cite[Lemma 1.a.6]{LT}),
we find $t=\{k_0<\dots<k_m\}\in [M]^{m+1}$ such that the sequence $(x_{k_i}^i)_{i=0}^m$ is basic
with basis constant $(1+\vep)$. Therefore, $\mathcal{B}(t)=0$, and by homogeneity,
$\mathcal{B}\upharpoonright [M]^{m+1}=0$. The claim is proved.
\end{proof}
Applying Claim \ref{o4ut44iui4} for $\vep=1$, we get an infinite subset $M$ of $\om$ as
described above. Observe that for every $\{n_0<\dots<n_m\}\in [M]^{m+1}$ and every choice
$(a_i)_{i=0}^m$ of scalars we have
\begin{equation}
\label{kjpoeijreoi}
\nrm{\sum_{i=0}^m a_i x_{n_i}^i}\ge \frac{1}{4} \max\set{|a_i|}{i=0,\dots,m}.
\end{equation}
The desired subset $L$ of $\om$ will be an infinite subset of $M$ obtained after another
application of Ramsey's theorem. Specifically, consider the coloring $\mathcal{U}:[M]^{m+1}\to 2$
defined as follows. Let $s=\{n_0<\dots<n_m\}\in [M]^{m+1}$ and assume that the sequence
$(x_{n_i}^i)_{i=0}^m$ is $(1+\vep)$-unconditional. In such a case, we set $\mathcal{U}(s)=0$;
otherwise we set $\mathcal{U}(s)=1$. Let $L$ be an infinite subset of $M$ be such $\mathcal{U}$
is constant on $[L]^{m+1}$. It is enough to find some $s\in [L]^{m+1}$ such that
$\mathcal{U}(s)=0$.

To this end, fix $\de>0$ such that $(1+\de)\cdot(1-\de)^{-1}\le (1+\vep)$. Notice that there exists a
finite family $\mathcal{D}$ of normalized Schauder basic sequences of length $m+1$ such that any
normalized Schauder basic sequence $(y_i)_{i=0}^m$, in some Banach space $Y$, is
$\sqrt{1+\de}$-equivalent to some sequence in the family $\mathcal{D}$. Hence, by a further
application of Ramsey's theorem and by passing to an infinite subset of $L$ if necessary,
we may assume that
\begin{enumerate}
\item[($\ast$)] for every $\{n_0<\dots<n_m\}, \{k_0<\dots<k_m\}\in [L]^{m+1}$ the
sequences $(x_{n_i}^i)_{i=0}^m$ and $(x_{k_i}^i)_{i=0}^m$ are $(1+\de)$-equivalent.
\end{enumerate}
Now, for every $i\in\{0,\dots,m\}$ and every $\rho>0$ let
\[ \mathcal{K}_i(\rho)=\big\{ \{n\in\om\; :\; |x^*(x_n^i)|\ge\rho\}\; :\; x^*\in B_{E^*}\big\}.\]
Every sequence $(x^i_n)$ is weakly null, and so, each $\mathcal{K}_i(\rho)$ is a
pre-compact\footnote[1]{Recall that a family $\mathcal{F}$ of finite subsets of $\om$ is said to be
pre-compact if, identifying $\mathcal{F}$ with a subset of the Cantor set $2^\om$, the closure
$\overline{\mathcal{F}}$ of $\mathcal{F}$ in $2^\om$ consists only of finite sets.}
family of finite subsets of $\om$. Hence, we may select a sequence $(F_i)_{i=0}^m$
of finite subsets of $L$ such that
\begin{enumerate}
\item[(a)] $\max(F_i)<\min(F_{i+1})$ for every $i\in\{0,\dots,m-1\}$, and
\item[(b)] $F_i\notin \mathcal{K}_i(\de\cdot 8^{-1}\cdot (m+1)^{-1})$ for every $i\in\{0,\dots,m\}$.
\end{enumerate}
We set $n_i=\min(F_i)$ for all $i\in\{0,\dots,m\}$. Property (a) above implies that
$n_0<\dots<n_m$. We claim that the sequence $(x_{n_i}^i)_{i=0}^m$ is $(1+\vep)$-unconditional.
Indeed, let $F\subseteq \{0,\dots,m\}$ and $(a_i)_{i=0}^m$ be a choice of scalars. We want to prove that
\[\nrm{\sum_{i\in F}a_i x_{n_i}^i}\le (1+\vep)\nrm{\sum_{i=0}^m a_ix_{n_i}^i}.\]
Clearly we may assume that $\nrm{\sum_{i\in F}a_i x_{n_i}^i}=1$. If
$\nrm{\sum_{i\notin F} a_i x_{n_i}^i}\ge 2$, then
\[ \nrm{\sum_{i=0}^m a_i x_{n_i}^i}\ge \nrm{\sum_{i\notin F}a_i x_{n_i}^i}-
\nrm{\sum_{i\in F}a_i x_{n_i}^i}\ge 1=\nrm{\sum_{i\in F}a_i x_{n_i}^i}.\]
So, suppose that $\nrm{\sum_{i\notin F} a_i x_{n_i}^i}\le 2$. By
\eqref{kjpoeijreoi}, we see that
\begin{equation}
\label{wirueirggf} \max\set{|a_i|}{i\notin F}\le 8.
\end{equation}
We select $x_0^*\in S_{E^*}$ such that $x_0^*\big(\sum_{i\in F}a_i x_{n_i}^i\big)=\nrm{\sum_{i\in F}a_i
x_{n_i}^i}$. We define a sequence $(k_i)_{i=0}^m$ in $L$ as follows. If $i\notin F$, then
let $k_i$ be any member of $F_i$ satisfying $|x_0^*(x_{k_i}^i)|<\de\cdot 8^{-1}\cdot(m+1)^{-1}$
(such a selection is possible by (b) above); if $i\in F$, then we set $k_i=n_i$. By (a),
we have $k_0<\dots <k_m$. Moreover,
\begin{eqnarray*}
\nrm{\sum_{i=0}^m a_i x_{k_i}^i} & =&  x_0^*\big(\sum_{i=0}^m a_i x_{k_i}^i\big) =
x_0^*\big(\sum_{i\in F} a_i x_{k_i}^i\big) +x_0^*\big(\sum_{i\notin F} a_i x_{k_i}^i\big) \\
& \geq & x_0^*\big(\sum_{i\in F} a_i x_{k_i}^i\big) - \sum_{i\notin F} |a_i|\cdot |x_0^*(x_{k_i}^i)|
\geq  1 -\de.
\end{eqnarray*}
Invoking ($\ast$), we conclude that
\[ \nrm{\sum_{i=0}^m a_i x_{n_i}^i}\ge \frac{1}{1+\de}\nrm{\sum_{i=0}^m a_i x_{k_i}^i}\ge
\frac{1-\de}{1+\de} \ge \frac{1}{1+\vep}\nrm{\sum_{i\in F}a_i x_{n_i}^i}. \]
The proof is completed.
\end{proof}
We are ready to proceed to the proof of Theorem \ref{mainappli}.
\begin{proof}[Proof of Theorem \ref{mainappli}]
Let $\ka$ be a cardinal such that $\mathrm{Pl}_2(\ka)$ holds. By a
classical result of R. C. James (see \cite[Proposition 2.e.3]{LT}),
it is enough to show that if $E$ is a Banach space of density $\ka$
not containing an isomorphic copy of $\ell_1$, then $E$ has an
$1$-unconditional basic sequence. So, let $E$ be one. By Rosenthal's
$\ell_1$ theorem \cite{Ro} and our assumptions on the space $E$, we see
that every bounded sequence in $E$ has a weakly Cauchy subsequence.
Let $(x_\al:\al<\ka)$ be a normalized sequence such that
$\nrm{x_\al-x_\be}\ge 1$ for every $\al<\be<\ka$. We define a
coloring $c_{\mathrm{un}}:\big[\,[\ka]^2\,\big]^{<\om}\to\om$ as
follows. Let $s=(\{\al_0<\be_0\},\dots,\{\al_m<\be_m\})\in
\big[\,[\ka]^2\,\big]^{<\om}$ arbitrary. Assume that there exists
$l\in\omega$ with $l>0$ and such that the sequence
$(x_{\be_i}-x_{\al_i})_{i=0}^m$ is \textit{not}
$(1+1/l)$-unconditional. In such a case, setting $l_s$ to be the
least $l\in\om$ with the above property, we define
$c_{\mathrm{un}}(s)=l_s$. If such an $l$ does not exist, then we set
$c_{\mathrm{un}}(s)=0$. By $\mathrm{Pl}_2(\ka)$, there exist a
sequence $(\xx_i)$ of infinite subsets of $\kappa$ and a sequence
$(l_m)$ in $\om$ such that for every $m\in\om$ the restriction
$c_{\mathrm{un}}\rest\prod_{i=0}^m[\xx_i]^2$ of the coloring
$c_{\mathrm{un}}$ on the product $\prod_{i=0}^m [\xx_i]^2$ is
constant with value $l_m$.
\begin{claim}
\label{newnewclaim} For every $m\in\om$ we have $l_m=0$.
\end{claim}
Grating the claim, the proof of the theorem is completed.
Indeed, observe that for every infinite sequence of
pairs $\big(\{\al_i<\be_i\}\big)\in \prod_{i\in\om} [\xx_i]^2$
the sequence $(x_{\be_i}-x_{\al_i})$ is a semi-normalized
$1$-unconditional basic sequence in the Banach space $E$.

It only remains to prove Claim \ref{newnewclaim}. To this
end we argue by contradiction. So, assume that there exists
$m\in\om$ such that $l_m>0$. Our definition of the coloring
$c_{\mathrm{un}}$ implies that $m\geq 1$. For every
$i\in\{0,\dots,m\}$ we may select an infinite subset
$\{\al^i_0<\al^i_1<\cdots\}$ of $\xx_i$ such that the sequence
$(x_{\al_i})$ is weakly Cauchy. We set
\[ y^i_n=\frac{x_{\al^i_{2n}}-x_{\al^i_{2n+1}}}{\|x_{\al^i_{2n}}-x_{\al^i_{2n+1}}\|} \]
for every $i\in\{0,\dots,m\}$ and every $n\in\om$. Then each $(y^i_n)$
is a normalized weakly null sequence in $E$. Moreover, for every
$\{n_0<\dots<n_m\}\subseteq [\nn]^{m+1}$ the sequence
$(y^i_{n_i})_{i=0}^m$ is not $(1+1/l_{m})$-unconditional. This
clearly contradicts Lemma \ref{wnullunco}. The proof is completed.
\end{proof}


\subsection{Unconditional subsequences of weakly null sequences}

This subsection is devoted to the proof of the following ``sequential" version of
Theorem \ref{mainappli}.
\begin{thm}
\label{thmnewnewnew} Let $\ka$ be a cardinal and assume that property $\mathrm{Pl}_1(\ka)$ holds
(see Definition \ref{oriutirt}).
Then $\mathfrak{nc}_{\mathrm{seq}}\leq \ka$. In fact, every
normalized weakly null sequence $(x_\al:\al<\ka)$ has an
$1$-unconditional subsequence.
\end{thm}
\begin{proof}
The proof is very similar to the one of Theorem \ref{mainappli}.
Indeed, consider the coloring $c_{\mathrm{un}}:[\ka]^{<\om}\to\om$
defined as follows. Let $s=(\al_0<\dots<\al_m)\in [\ka]^{<\om}$.
Assume that there exists $l\in\om$ with $l>0$ such that the sequence
$(x_{\al_i})_{i=0}^m$ is \textit{not} $(1+1/l)$-unconditional. In
such a case, let $c_{\mathrm{un}}(s)$ be the least $l$ with this
property. Otherwise, we set $c_{\mathrm{un}}(s)=0$. Using
$\mathrm{Pl}_1(\ka)$ and Lemma \ref{wnullunco}, the result follows.
\end{proof}
\begin{cor}
\label{cor123456} It is  consistent relative to the existence of a single measurable cardinal
that every normalized weakly null sequence $(x_\al:\al<\aleph_\om)$ has an $1$-unconditional
subsequence. Moreover, this statement is consistent with GCH.
\end{cor}
\begin{proof}
Follows immediately by Theorem \ref{thmnewnew} and Theorem \ref{thmnewnewnew}.
\end{proof}
There is another well-known combinatorial property of a cardinal
$\ka$ which is implied by $\pl_1(\ka)$and which is in turn
sufficient for the estimate $\mathfrak{nc}_{\mathrm{seq}}\leq
\ka$. This property is in the literature called  the \textit{free
set property} of $\ka$ (see \cite{Sh1}, \cite{Ko}, \cite{DT} and
the references therein).
\begin{defn}\label{freesetproperty}
By a \emph{structure} on $\ka$ we mean a first order structure
$\mathcal{M}=(\ka,(f_i)_{i\in\om})$, where $n_i\in\om$ and
$f_i:\ka^{n_i}\to \ka$ for all $i\in\om$.

The free set property of $\ka$, denoted by $\mathrm{Fr}_\om(\ka,\om)$,
is the assertion that every structure $\mathcal{M}=(\ka,(f_i)_{i\in\om})$
has a free infinite set. That is, there exists an infinite subset $L$ of
$\ka$ such that every element $x$ of $L$ does not belong to the
substructure of $\mathcal{M}$ generated by $L\setminus\{x\}$.
\end{defn}
We need the following fact (its proof is left to the interested
reader).
\begin{fact}
Let $\ka$ be a cardinal. Then the following are equivalent.
\begin{enumerate}
\item[(a)] $\mathrm{Fr}_\om(\ka,\om)$ holds.
\item[(b)] For every structure $\mc M=(\ka,(f_i)_{i\in\om})$ there
exists an infinite subset $L$ of $\ka$ such that for every
$x\in L$ we have
\begin{equation}\nonumber
x\notin \set{f_i(s)}{s\in (L\setminus \{x\})^{n_i} \text{ and } i\in \om }.
\end{equation}
\item[(c)] Every \emph{extended} structure $\mathcal{N}=(\ka,(g_i)_{i\in\om})$,
where $g_i:\ka^{<\om}\to [\ka]^{\le \om}$ for all $i\in\om$, has an
infinite free subset. That is, there exists an infinite subset $L$
of $\ka$ such that for every $x\in L$ we have
\begin{equation}\nonumber
x\notin \bigcup_{i\in \om}\bigcup_{s\in (L\setminus \{x\})^{<\om}} g_i(s).
\end{equation}
\end{enumerate}
\end{fact}
As we have already indicated above, one can use the property $\mathrm{Fr}_\om(\ka,\om)$ to derive
the conclusion of Theorem \ref{thmnewnewnew}. More precisely, we have the following.
\begin{thm}
\label{mainappli2} Let $\ka$ be a cardinal and assume that
$\mathrm{Fr}_\om(\ka,\om)$ holds. Then every normalized weakly null
sequence $(x_\al:\al<\ka)$ has an $1$-unconditional subsequence.
\end{thm}
\begin{proof}
Let $(x_\al:\al<\ka)$ be a normalized weakly null sequence in
a Banach space $E$.  For every $s\in [\ka]^{<\om}$ we select
a subset $F_s$ of $S_{E^*}$ which is countable and 1-norming
for the finite-dimensional subspace $E_s:=\mathrm{span}\{x_\al:\al\in s\}$
of $E$. That is, for every $x\in E_s$ we have
\begin{equation}\label{jeieef}
\nrm{x}=\sup\set{x^*(x)}{x\in F_s}.
\end{equation}
Define $g:[\ka]^{<\om}\to [\ka]^{\le \om}$ by
\begin{equation}
g(s)=\set{\al<\ka}{\text{there is some $x^*\in F_s$ such that $x^*(x_\al)\neq 0$}}.
\end{equation}
Since $(x_\al:\al<\ka)$ is weakly null and $F_s$ is countable, we
see that $g(s)$ is also countable; i.e. $g$ is well-defined.
Consider the extended structure $\mathcal{N}=(\ka,g)$. Since
$\mathrm{Fr}_\om(\ka,\om)$ holds, there exists an infinite free
subset $L$ of $\ka$. We claim that the sequence $(x_\al:\al\in L)$
is 1-unconditional.

Indeed, let $s$ and $t$ be finite subsets of $L$ with $s\subseteq
t$. Fix a sequence $(a_\al:\al\in t)$ of scalars and let $\vep>0$
arbitrary. By equality (\ref{jeieef}) above, we may select $y^*\in
F_s$ such that
\begin{equation}
\nrm{\sum_{\al\in s} a_\al x_\al}\le (1+\vep) \cdot
y^*\big(\sum_{\al\in s} a_\al x_\al\big).
\end{equation}
The set $L$ is free, and so, for every $\al\in t\setminus s$ we have
$\al\notin g(s)$. This implies, in particular, that $y^*(x_\al)=0$
for every $\al\in t\setminus s$. Hence
\begin{eqnarray*}
\nrm{\sum_{\al\in s} a_\al x_\al} & \le &  (1+\vep)\cdot
y^*\big(\sum_{\al\in s}
a_\al x_\al\big) =(1+\vep) \cdot y^*\big(\sum_{\al\in t} a_\al x_\al\big)\\
& \le & (1+\vep) \cdot \nrm{\sum_{\al\in t} a_\al x_\al}.
\end{eqnarray*}
Since $\vep>0$ was arbitrary, the result follows.
\end{proof}


\section{Proof of Lemma \ref{MainLemma}}

Assume that $\la<\ka$ is a pair of two infinite cardinals with $\la$ regular and $\ka$ strongly compact.
Let $G$ be a $\col$-generic filter. \textit{The generic filter $G$ will be fixed until the end of
the proof. We also fix a $\ka$-complete normal ultrafilter $\mathcal{U}$ on $\ka$}.

Let $d\in\om$ with $d\geq 1$ arbitrary. Let $\set{V_\al}{\al\in\mathrm{Ord}}$ be the von-Neumann hierarchy
of $V$. As $\ka$ is inaccessible (being strongly compact), we see that $|V_\ka|=\ka$.
For every coloring $c:[(\expp_d(\ka))^+]^{d+1}\to  V_\ka$ we let
\begin{equation}
\label{esol1} \mathrm{Sol}^{\om}_{d,\ka}(c)= \set{\xx\in [(\expp_d(\ka))^+]^\om}{c\rest [\xx]^{d+1}
\text{ is constant}}
\end{equation}
and we define
\begin{equation}
\label{esol2} \mathrm{Sol}^{\om}_{d,\ka}= \set{\mathrm{Sol}^{\om}_{d,\ka}(c)}{c:[(\expp_d(\ka))^+]^{d+1}
\to  V_\ka\text{ is a coloring}}.
\end{equation}
The idea of considering the family of sets which are monochromatic with respect to a coloring is taken
from Shelah's paper \cite{Sh2} and has been also used by other authors (see, for instance, \cite{Mi}).
\begin{fact}
\label{fsol1} The following hold.
\begin{enumerate}
\item[(a)] For every coloring $c:[(\expp_d(\ka))^+]^{d+1}\to V_\ka$ we have
$\mathrm{Sol}^{\om}_{d,\ka}(c)\neq \emptyset$.
\item[(b)] The family $\mathrm{Sol}^\om_{d,\ka}$ is $\ka$-complete. That is, for every
$\de<\ka$ and every sequence $(A_\xi:\xi<\de)$ in $\mathrm{Sol}^\om_{d,\ka}$ we
have that $\bigcap_{\xi<\de}A_\xi\in \mathrm{Sol}^\om_{d,\ka}$.
\end{enumerate}
\end{fact}
\begin{proof}
(a) By our assumptions we see that $|V_\ka|=\ka$. Moreover, by the
classical Erd\H{o}s-Rado partition Theorem (see \cite{Ku}), we have
\[ (\expp_d(\ka))^+ \rightarrow (\ka^+)^{d+1}_\ka \]
and the result follows. \\
(b) For every $\xi<\de$ let $c_\xi: [\expp_d(\ka))^+]^{d+1}\to  V_\ka$ be a
coloring such that $A_\xi=\mathrm{Sol}^\om_{d,\ka}(c_\xi)$. Observe that
$(V_\ka)^\de\subseteq V_\ka$. We define the coloring
$c:[\expp_d(\ka))^+]^{d+1}\to (V_\ka)^\de$ by $c(s)=(c_\xi(s):\xi<\de)$.
Noticing that
\[ \bigcap_{\xi<\de}\mathrm{Sol}^\om_{d,\ka}(c_\xi)=\mathrm{Sol}^\om_{d,\ka}(c)\]
the proof is completed.
\end{proof}
By Fact \ref{fsol1}(b) and our hypothesis that $\ka$ is a strongly compact cardinal, we see that there
exists a $\ka$-complete ultrafilter $\vvv$ on $[\expp_d(\ka))^+]^\om$ extending the family
$\mathrm{Sol}^\om_{d,\ka}$. \textit{We fix such an ultrafilter $\mathcal{V}$}.
\begin{defn}
\label{dcanonical1} A \emph{$\vvv$--sequence of conditions} is a sequence $\overp=(p_\xx:\xx\in A)$ in
$\coll(\lambda,<\kappa)$, belonging to the ground model $V$ and  indexed by a member $A$ of the ultrafilter
$\vvv$. We will refer to the set $A$ as the \emph{index set} of $\overp$ and we shall denote it by
$I(\overp)$.
\end{defn}
\begin{defn}
\label{dcanonical2} Let $\overp=(p_\xx:\xx\in I(\overline{p}))$ be a $\vvv$--sequence of conditions.
We say that a condition $r$ in $\coll(\lambda,<\kappa)$ is a \emph{root} of $\overp$ if
\begin{equation}
\label{e1} (\uuu\al) \ (\vvv\xx) \ \ p_\xx\upharpoonright
\al=r\footnote[2]{This is an abbreviation of the statement that
$\set{\al}{\set{\xx}{p_\xx\upharpoonright \al=r}\in\vvv}\in\uuu$.}.
\end{equation}
\end{defn}
Related to the above definitions, we have the following.
\begin{fact}
\label{fcanonical} Every $\vvv$--sequence of conditions $\overline{p}$ has a unique root
$r(\overline{p})$.
\end{fact}
\begin{proof}
For every $\al<\ka$ the map $I(\overline{p})\ni\xx\mapsto p_\xx\rest\al$ has fewer than $\ka$ values.
So, by the $\ka$-completeness of $\mc V$, there exist $p_\al\in \col$ and
$I_\al\in \mc V\rest I(\overline{p})$ so that $p_\xx\rest \al=p_\al$
for all $\xx\in I_\al$. Hence, we can select a sequence
$(p_\al:\al<\ka)$ in $\col$ and a decreasing sequence
$(I_\al:\al<\ka)$ of elements of $\mc V\rest I(\overline{p})$ such
that for every $\al<\ka$ and every $\xx\in I_\al$ we have that
$p_\xx\rest \al=p_\al$.

Let $A\sset \ka$ be the set of all limit ordinals $\al<\ka$ with
$\mathrm{cf}(\al)>\la$. Since $\mc U$ is normal, the set $A$ is in
$\mc U$. Consider the mapping $c:A\to\ka$ defined by
\[ c(\al)= \sup\set{\xi}{\xi\in(\mathrm{dom}(p_\al\rest \al))_1}\]
for every $\al\in A$. As $\mathrm{cf}(\al)>\la$, we get that $c$ is
a regressive mapping. The ultrafilter $\mc U$ is normal, and so,
there exist $A'\in \mc U\rest A$ and $\ga_0<\ka$ such that
$c(\al)=\ga_0$ for every $\al\in A'$. Now consider the map
\[ A'\ni\al\mapsto p_\al\rest\al=p_\al\rest\ga_0\sset (\la \times \ga_0)\times \ga_0.\]
Noticing that
$|\mathcal{P}\big((\la\times\ga_0)\times\ga_0\big)|<\ka$ and
recalling that $\uuu$ is $\ka$-complete, we see that there exist
$A''\in \mc U\rest A'$ and $r(\overline{p})$ in $\col$ such that
$p_\al\rest\al =r(\overline{p})$ for every $\al\in A''$. It follows
that for every $\al\in A''$ the set $\set{\xx\in [\expp_d(\ka))^+]^\om}{p_\xx\rest\al=r(\overline{p})}$
contains the set $I_\al$, and so
\[ (\uuu\al) \ (\vvv\xx) \ \ p_\xx\upharpoonright \al=r(\overline{p}). \]
The uniqueness of $r(\overline{p})$ is an immediate consequence of
property (\ref{e1}) in Definition \ref{dcanonical2}. The proof is
completed.
\end{proof}
We are ready to introduce the ideal $\iii_d$.
\begin{defn}\label{defideal}
In $V[G]$ we define
\[ \iii_d=\set{I\sset [\expp_d(\ka))^+]^\om}{\text{ there is some }
A\in \mc V \text{ such that } I\cap A=\emptyset}.\]
\end{defn}
We isolate, for future use, the following (easily verified) properties of $\iii_d$.
\begin{enumerate}
\item[(P1)] $\iii_d$ is an ideal; in fact, $\iii_d$ is a $\ka$-complete ideal.
\item[(P2)] $\vvv\sset \iii_d^+$.
\item[(P3)] If $A\in\vvv$ and $B\in \iii_d^+$, then $A\cap B\in \iii_d^+$.
\end{enumerate}
For every $\vvv$--sequence of conditions $\overline{p}$ we let
\begin{equation}
\label{dp} D_{\overline p}=\set{\xx\in I(\overline p)}{p_\xx \in G}.
\end{equation}
Now we are ready to introduce the set $\ddd_d$.
\begin{defn}
\label{defdense}
In $V[G]$ we define
\[ \ddd_d=\set{D_{\overline p}}{\overline p \text{ is a $\vvv$--sequence of conditions in
the ground model $ V$}}\cap \iii_d^+.\]
\end{defn}
By definition, we have that $\ddd_d\subseteq \iii_d^+$. The rest of the
proof will be devoted to the verification that the ideal $\iii_d$ and
the set $\ddd_d$ satisfy the requirements of Lemma \ref{MainLemma}. To
this end, we need the following.
\begin{lem}
\label{characterization} Let $\overline p=(p_\xx:\xx\in I(\overline{p}))$ be a $\vvv$--sequence
of conditions. Then the following are equivalent.
\begin{enumerate}
\item[(1)] $D_{\overline p}\in \ddd_d$.
\item[(2)] $r(\overline{p})\in G$.
\end{enumerate}
\end{lem}
\begin{proof}
(1)$\Rightarrow$(2) Assume that $D_{\overline p}\in \ddd_d$. We use the fact that
$D_{\overline{p}}\in \iii_d^+$ and that
\[ (\mc U \al) \ (\mc V \xx) \ \ p_\xx \rest \al=r(\overline p)\]
to find $\xx\in D_{\overline p}$ such that $p_\xx \le r(\overline{p})$. By the definition of
$D_{\overline p}$, we see that $p_{\xx}\in G$,
and so, $r(\overline p)\in G$ as well.\\
(2)$\Rightarrow$(1) Suppose that $r(\overline{p})\in G$. Fix a
ground model set $A$ which is in $\mc V$. It is enough to show that
$D_{\overline p}\cap A\neq \emptyset$. To this end, let
\[ E=\set{q\in \col}{q \perp r(\overline p)\text{ or there is }
\xx\in I(\overline p)\cap A \text{ with } q\le p_\xx}.\] We claim
that $E$ is a dense subset of $\col$. To see this, let $r\in \col$
arbitrary. If $r\perp r(\overline p)$, then $r\in E$. So, suppose
that $r\parallel r(\overline p)$. Using this and the fact that
\[ (\mc U \al) \ (\mc V \xx) \ \ p_\xx \rest \al=r(\overline p)\]
we may find $\xx\in I(\overline p)\cap A$ such that $p_\xx\ \| \ r$.
So, there exist $q\in\col$ and $\xx\in I(\overline{p})\cap A$ such
that $q\leq p_\xx$ and $q\leq r$. In other words, there exists $q\in
E$ with $q\leq r$. This establishes our claim that $E$ is a dense
subset of $\col$.

It follows by the above discussion that there exists $q\in G$ with
$q\in E$. Since $r(\overline p)\in G$ we have that $r(\overline{p})
\ \| \ q$. Hence, by the definition of the set $E$, there exists
$\xx\in I(\overline p)\cap A$ with $q\le p_\xx$. It follows that
$p_\xx\in G$, and so, $\xx\in D_{\overline p}\cap A$. The proof is
completed.
\end{proof}
\begin{lem}
\label{denseness} $\ddd_d$ is dense in $\iii_d^+$.
\end{lem}
\begin{proof}
Fix $J\in \iii_d^+$. We will prove that there exists a $\vvv$--sequence of conditions $\overline q$
in the ground model $V$ satisfying $D_{\overline{q}}\in\ddd_d$ and $D_{\overline q}\sset J$.
This will finish the proof.

To this end, we fix a $\col$-name $\dot{J}$ for $J$. Let $p\in \col$
be an arbitrary condition such that $p\Vdash \dot{J} \notin \iii_d$.
Define, in the ground model $V$, the set
\[ A_p=\set{\xx\in [\expp_d(\ka))^+]^\om}{\text{there is }q\le p \text{ such that }
q \Vdash \check{\xx}\in \dot{J}}.\]
First we claim that $A_p\in \vvv$. Suppose, towards a contradiction, that the set
$C:=[\expp_d(\ka))^+]^\om \setminus A_p$ is in $\vvv$. Since $J\in\iii_d^+$ we see that
$J\cap C\neq\emptyset$ in $V[G]$. Using the fact that $p\Vdash \dot{J}\notin \iii_d$
and that the forcing $\col$ is $\sigma$-closed, we may find $\xx\in C$ and a condition $q\leq p$
such that $q\Vdash \check{\xx}\in\dot{J}$. But this implies that $\xx\in A_p$, a contradiction.

It follows that we may select a $\vvv$--sequence of conditions
$\overline{q}=(q_\xx:\xx\in A_p)$ such that $q_\xx\leq p$ and $q_\xx
\Vdash \check{\xx}\in \dot{J}$ for every $\xx\in A_p$. By Fact
\ref{fcanonical}, let $r(\overline{q})$ be the root of
$\overline{q}$. Clearly $r(\overline{q})\leq p$.

Now, fix a condition $r$ such that $r\Vdash \dot{J} \notin \iii_d$.
What we have just proved is that the set of conditions $r(\overline{q})$ such that
\begin{enumerate}
\item[($\ast$)] $r(\overline{q})$ is the root of a $\vvv$--sequence of conditions
$\overline{q}=(q_\xx:\xx\in I(\overline{q}))$ with the property that
$q_\xx\Vdash \check{\xx}\in \dot{J}$ for every $\xx\in I(\overline{q})$
\end{enumerate}
is dense below $r$. As $G$ is generic, we see that there exists a $\vvv$--sequence of
conditions $\overline{q}$ as in ($\ast$) above such that $r(\overline{q})\in G$. On the
one hand, by Lemma \ref{characterization}, we see that $D_{\overline{q}}\in \ddd_d$.
On the other hand, property ($\ast$) above implies that $D_{\overline{q}}\subseteq J$;
indeed, if $\xx\in D_{\overline{q}}$, then $q_\xx\in G$ and, by ($\ast$),
$q_\xx\Vdash \check{\xx}\in\dot{J}$. The proof is completed.
\end{proof}
\begin{lem}
\label{la+closed} $\ddd_d$ is $\la$-closed in $\iii_d^+$.
\end{lem}
\begin{proof}
Fix $\mu<\la$ and a decreasing sequence $(D_\xi:\xi<\mu)$ in $\ddd_d$.
For every $\xi<\mu$ let $\overline{p}_\xi=(p^\xi_\xx:\xx\in I(\overline{p}_\xi))$ be a
$\vvv$--sequence of conditions in $V$ such that $D_\xi=D_{\overline{p}_\xi}$. Our forcing
$\col$ is $\la$-closed, and so, the sequence $(\overline{p}_\xi:\xi<\mu)$ is in the ground
model $V$ as well. Applying Fact \ref{fcanonical} to every $\overline{p}_\xi$, we find a
sequence $(r_\xi:\xi<\mu)$ in $\col$ such that $r_\xi$ is the root of $\overline{p}_\xi$ for
every $\xi<\mu$. By Lemma \ref{characterization}, we get that $r_\xi\in G$ for all $\xi<\mu$.

We claim, first, that for every $\xi<\zeta<\mu$ we have
\begin{equation}
\label{enew1} (\mc V \xx) \ \ p^\xi_\xx \parallel p^{\zeta}_\xx.
\end{equation}
Suppose, towards a contradiction, that there exist $\xi<\zeta<\mu$
such that the set $L:=\set{\xx \in A}{p^\xi_\xx  \perp p^{\zeta}_\xx}$ is in $\vvv$.
As $D_{\overline{p}_\zeta}\in \ddd_d\subseteq \iii_d^+$ and $L\in\vvv$, there exists
$\xx\in D_{\overline{p}_\zeta}\cap L$. And since
$D_{\overline{p}_\zeta}=D_\zeta\sset D_{\xi}=D_{\overline{p}_\xi}$
we have $\xx\in D_{\overline{p}_\xi}$ as well. But this implies that
both $p^\xi_\xx $ and $p^{\zeta}_\xx$ are in $G$ and at the same
time $p^\xi_\xx  \perp p^{\zeta}_\xx$, a contradiction.

Invoking (\ref{enew1}) above, we may find $A\in\vvv$ such that for
every $\xi<\zeta<\mu$ and every $\xx\in A$ we have that $p^{\xi}_\xx
\parallel p^{\zeta}_{\xx}$. We set
\[ p_\xx=\bigcup_{\xi<\mu} p^\xi_\xx \ \ \text{ for every } \xx\in A\]
and we define $\overline{p}=(p_\xx:\xx\in A)$. It is clear that
$\overline{p}$ is a well-defined $\vvv$--sequence of conditions. Also
observe that $D_{\overline{p}}\subseteq D_\xi$ for every $\xi<\mu$.
We are going to show that $D_{\overline{p}}\in\ddd_d$. This will
finish the proof.

To this end, let $r$ be the root of $\overline{p}$. By Lemma \ref{characterization},
it is enough to show that $r\in G$. Notice, first, that
\begin{equation}
\label{hgjhj0} (\uuu\al) \ (\vvv\xx) \ \ \bigcup_{\xi<\mu} p^\xi_\xx\rest\al= p_\xx\rest\al =r.
\end{equation}
On the other hand, as $r_\xi$ is the root of $\overline{p}_\xi$, we have
\begin{equation}
\label{hgjhj1} (\forall \xi<\mu) \ (\uuu\al) \ (\vvv\xx) \ \ p_\xx^\xi \rest\al =r_\xi.
\end{equation}
Both $\uuu$ and $\vvv$ are $\ka$-complete, and so, (\ref{hgjhj1}) is equivalent to
\begin{equation}
\label{hgjhj2} (\uuu\al) \ (\vvv\xx) \ (\forall \xi<\mu) \ \ p_\xx^\xi \rest \al=r_\xi.
\end{equation}
Combining (\ref{hgjhj0}) and (\ref{hgjhj2}) we get that
\begin{equation}
\label{hgjhj3} (\uuu\al) \ (\vvv\xx) \ \ r=\bigcup_{\xi<\mu} p_\xx^\xi \rest \al =\bigcup_{\xi<\mu} r_\xi.
\end{equation}
Summing up, we see that the root $r$ of $\overline p$ is the union
$\bigcup_{\xi<\mu} r_\xi$ of the roots of the $\overline{p}_\xi$'s.
Since the generic filter $G$ is $\la$-complete, we conclude that
$r\in G$. The proof is completed.
\end{proof}
\begin{lem}
\label{colors} Work in $V[G]$. Let $\mu<\ka$ and let $c:[\expp_d(\ka))^+]^{d+1}\to \mu$ be a coloring.
Let also $A\in \iii_d^+$ arbitrary. Then there exist a color $\xi<\mu$ and an element $D\in\ddd_d$
with $D\sset A$ and such that for every $\xx\in D$ and every $\{\al_0,\dots,\al_d\} \in[\xx]^{d+1}$
we have $c(\{\al_0,\dots,\al_d\})=\xi$.
\end{lem}
\begin{proof}
Fix a coloring $c:[\expp_d(\ka))^+]^{d+1}\to \mu$ and let $A\in\iii_d^+$. Let also $\dot{c}$
be a $\col$-name for the coloring $c$. In $V$, let $\mathrm{RO}(\col)$ be the collection of all
regular-open subsets of $\col$. Working in $V$, we define another coloring
$C:[\expp_d(\ka))^+]^{d+1}\to (\mathrm{RO}(\col))^{\mu}$ by the rule
\[ C(s)=([\hspace{-0.05cm}[\dot{c}(\check{s})=\check{\xi} ]\hspace{-0.045cm}]:\xi<\mu)\]
where $[\hspace{-0.05cm}[\dot{c}(\check{s})=\check{\xi}  ]\hspace{-0.045cm}]=
\set{p\in \col}{p\Vdash\dot{c}(\check{s})=\check{\xi}}$ is the boolean
value of the formula ``$c(s)=\xi$".

The forcing $\col$ is $\ka$-cc, and so, $(\mathrm{RO}(\col))^{\mu}\sset {V}_\ka$. Hence,
$\mathrm{Sol}^\om_{d,\ka}(C)\in\vvv$. We set $J=A\cap \mathrm{Sol}^\om_{d,\ka}(C)$.
Then $J$ is in $\iii_d^+$. Notice that for every $\xx\in J$ and every $s, s'\in [\xx]^{d+1}$
we have $C(s)=C(s')$. It follows that for every $\xx\in J$ we may select a sequence
$\overline{U}_\xx=(U^\xi_{\xx}:\xi<\mu)$ in $(\mathrm{RO}(\col))^{\mu}$ such that for every
$s\in [\xx]^{d+1}$ and every $\xi<\mu$ we have
$[\hspace{-0.05cm}[\dot{c}(\check{s})=\check{\xi}]\hspace{-0.045cm}]=U^\xi_{\xx}$.

Now observe that for every $s\in [\expp_d(\ka))^+]^{d+1}$ the set
\[ \set{[\hspace{-0.05cm}[\dot{c}(\check{s})=\check{\xi}]\hspace{-0.045cm}]}{\xi<\mu} \]
is a maximal antichain. So, we can naturally define in $V[G]$ a coloring $e:J\to \mu$ by the rule
\[ e(\xx)=\xi \ \text{ if and only if } \ U^\xi_\xx\in G.\]
Equivalently, for every $\xx\in J$ we have that $e(\xx)=\xi$ if and
only if $c\rest [\xx]^{d+1}$ is constant with value $\xi$. The ideal
$\iii_d$ is $\ka$-complete and $J\in \iii_d^+$. Hence there exists
$\xi_0<\mu$ such that $e^{-1}\{\xi_0\}\in \iii_d^+$. By Lemma
\ref{denseness}, we may select $D\in \ddd_d$ with $D\sset
e^{-1}\{\xi_0\}\sset J\sset A$. Finally, notice that for every
$\xx\in D$ the restriction $c\rest [\xx]^{d+1}$ is constant with value
$\xi_0$. The proof is completed.
\end{proof}
We are ready to finish the proof of Lemma \ref{MainLemma}. As we have already mention, the ideal
$\iii_d$ will be the one defined in Definition \ref{defideal}, while the dense subset $\ddd_d$
of $\iii_d^+$ will be the one defined in Definition \ref{defdense}. First, we notice that property
(1) in Lemma \ref{MainLemma} (i.e. the fact that $\iii_d$ is $\ka$-complete) follows easily by the
definition $\iii_d$ and the fact that $\vvv$ is $\ka$-complete (in fact, we have already isolated
this property of $\iii_d$ in (P1) above). Property (2) in Lemma \ref{MainLemma} (i.e. the fact that
$\ddd_d$ is $\la$-closed in $\iii_d^+$) has been established in Lemma \ref{la+closed}. Finally,
property (3) was proved in Lemma \ref{colors}. Since $d\geq 1$ was arbitrary, the proof of Lemma
\ref{MainLemma} is completed.

\section{Concluding remarks}

In this section we would like to discuss the possible refinements
of our Theorem \ref{ithm1}. First of all we notice that
Ketonen's arguments actually give that if the density of a given
Banach space $E$ is greater or equal than the $\om$-Erd\H{o}s
cardinal, then $E$ contains a normalized basic sequence which is
equivalent to all of its subsequences, i.e. a  basic sequence
which is in the literature usually called a \emph{sub-symmetric
basic sequence}. Note that this is stronger than saying that the
space $E$ contains an unconditional basic sequence which can be
easily seen using Rosenthal's $\ell_1$ theorem \cite{Ro}.

On the other hand, we notice that our proof of the existence of an
unconditional basic sequence in every Banach space of density
$\expp_\om(\aleph_0)$ \textit{does not} guarantee the existence of a
sub-symmetric basic sequence. This is mainly due to the fact that
the principle $\pl_2(\ka)$ is a rectangular Ramsey property while
all attempts that we have in mind for getting sub-symmetric basic
sequences seem to require more classical Ramsey-type principles such
as these given, for example, by the $\om$-Erd\H{o}s cardinal. Since
$\om$-Erd\H{o}s is a large-cardinal property one might expect that
there are Banach spaces of large density not containing a
sub-symmetric basic sequence. So let us discuss some difficulties
one encounters when trying to build such spaces.

The first example of an infinite dimensional Banach space not
containing a sub-symmetric basic sequence is Tsirelson's space
\cite{Ts}. Tsirelson's space is separable; however, there do exist
non-separable Banach spaces with the same property. The first such
example is due to E. Odell \cite{O1}. Odell's space is the dual of
a separable one, and so, it has density $2^{\aleph_0}$. There even
exist non-separable \emph{reflexive} spaces not containing a
sub-symmetric basic sequence. For example, one such a space is
the space constructed in \cite{ALT} which has density $\aleph_1$.
We note that both spaces of \cite{O1} and of \cite{ALT} are
connected in some way to the Tsirelson space. So one is led to
explore generalizations of the Tsirelson construction to larger
densities.

Let us comment on difficulties encountered when trying to generalize
Tsirelson's construction to densities bigger than the continuum,
keeping in mind that we would like to get a space not containing a
sub-symmetric basic sequence. The first natural move is to provide,
for a given cardinal $\ka$, a compact hereditary family
$\mathcal{F}$ of finite subsets of $\ka$ which is sufficiently rich
in the sense that for every infinite subset $M$ of $\ka$ the
restriction $\mathcal{F}\upharpoonright M$ of the family on $M$ has
infinite rank. Notice that such a family cannot exist if $\ka$ is
greater or equal the $\om$-Erd\H{o}s cardinal. On the other hand,
using a characterization of $n$-Mahlo cardinals due to J. H. Schmerl
(see \cite{sch} or \cite[Theorem 6.1.8]{tod}), we were able to show
that if $\ka$ is smaller that the first $\om$-Mahlo cardinal, then
$\ka$ carries such a family $\mathcal{F}$.

Given a compact hereditary family $\mathcal{F}$ as above, the next
step is to construct the Tsirelson-like space $T(\mathcal{F})$ on
$c_{00}(\ka)$ in the natural way. Such a space always fails to
contain $c_0$ and $\ell_p$ for any $1<p<\infty$. However, there are
examples of such families for which the corresponding space contains
a copy of $\ell_1$. The reason is that the family $\mathcal{F}$
cannot be \emph{spreading} relative the natural well-ordering of
ordinals if $\ka$ is uncountable. Recall that spreading is a crucial
property of the Schreier family on $\omega$ used in the original
Tsirelson construction for preventing isomorphic copies of $\ell_1$.
We are grateful to Spiros A. Argyros for pointing out this to us
after reading a previous version of this paper containing the
erroneous claim that $T(\mathcal{F})$ contains no isomorphic copy of
$\ell_1$.

In fact, in order to prevent the embedding of $\ell_1$ inside $T(\mathcal{F})$ it suffices,
beside the above requirements, to ensure that the family $\mathcal{F}$ is \textit{weak spreading}
in the sense that if $\al_0\leq\be_0\leq\al_1\leq\dots\leq\al_n\leq\be_n$ and $\{\al_0,\dots,\al_n\}$
is in the family $\mathcal{F}$, then $\{\be_0,\dots,\be_n\}$ is also in $\mathcal{F}$. It is unclear
to us whether such a family $\mathcal{F}$ can exist on a cardinal $\ka$ greater than $\aleph_1$.


\end{document}